\documentclass[a4paper,english]{amsart}
\usepackage[latin9]{inputenc}
\usepackage{tikz}
\usepackage{caption}
\usepackage{amsmath}
\usepackage{bbm}
\usepackage{babel}
\usepackage{verbatim}
\usepackage{mathtools}
\usepackage{amsthm}
\usepackage{amstext}
\usepackage{amssymb}
\usepackage{amsfonts}
\usepackage{esint}
\usepackage[unicode=true,
bookmarks=false,
breaklinks=false,pdfborder={0 0 1},backref=false,colorlinks=false]
{hyperref}
\hypersetup{
	colorlinks,linkcolor=blue,anchorcolor=blue,citecolor=blue}
\usepackage{breakurl}
\usepackage{mathrsfs}

\makeatletter

\numberwithin{equation}{section}
\numberwithin{figure}{section}
\usepackage{enumitem}		% customizable list environments

\theoremstyle{plain}
\newtheorem{thm}{Theorem}[section]
\theoremstyle{plain}
\newtheorem{prop}[thm]{Proposition}
\theoremstyle{remark}
\newtheorem{rem}[thm]{Remark}

\theoremstyle{plain}

\theoremstyle{plain}

\theoremstyle{plain}
\newtheorem{lem}[thm]{Lemma}
\theoremstyle{definition}

\theoremstyle{definition}

\theoremstyle{definition}
\newtheorem{con}{Conjecture}

\newtheorem*{thmA*}{Theorem A}
\newtheorem*{thmA1*}{Theorem A1}
\newtheorem*{thmA2*}{Theorem A2}
\newtheorem*{thmA3*}{Theorem A3}

\newcommand{\real}{\mathbb{R}}
\newcommand{\Tr}{\textnormal{Tr}}

\newcommand{\bh}{\mathcal{B(H)}}
\newcommand{\ph}{\mathcal{P(H)}}
\renewcommand{\dh}{\mathcal{D(H)}}
\newcommand{\bhx}{\mathcal{B(H)}^{\times}}
\newcommand{\phx}{\mathcal{P(H)}^{\times}}
\newcommand{\dhx}{\mathcal{D(H)}^{\times}}
\newcommand{\h}{\mathcal{H}}
\renewcommand{\a}{\alpha}

\newcommand{\vertiii}[1]{{\left\vert\kern-0.25ex\left\vert\kern-0.25ex\left\vert #1 
		\right\vert\kern-0.25ex\right\vert\kern-0.25ex\right\vert}}

\begin{document}
\title{From Wigner-Yanase-Dyson conjecture to Carlen-Frank-Lieb conjecture}
\author{Haonan Zhang}

\address{Laboratoire de Math\'ematiques, Universit\'e Bourgogne Franche-Comt\'e, 25030 Besan\c con, France and Institute of Mathematics, Polish Academy of Sciences, ul. \'Sniadeckich 8, 00-656 Warszawa, Poland}
\address{Current address: Institute of Science and Technology Austria (IST Austria), Am Campus 1, 3400 Klosterneuburg, Austria}

\email{haonan.zhang@ist.ac.at}

\subjclass[2010]{Primary 15A15, 81P45; Secondary 47A56, 94A17}

\keywords{Joint convexity/concavity, quantum relative entropy, Data Processing Inequality}
\maketitle

\begin{abstract}

In this paper we study the joint convexity/concavity of the trace functions
\[
\Psi_{p,q,s}(A,B)=\Tr(B^{\frac{q}{2}}K^*A^{p}KB^{\frac{q}{2}})^s,~~p,q,s\in \mathbb{R},
\]
where $A$ and $B$ are positive definite matrices and $K$ is any fixed invertible matrix. We will give full range of $(p,q,s)\in\mathbb{R}^3$ for $\Psi_{p,q,s}$ to be jointly convex/concave for all $K$. As a consequence, we confirm a conjecture of Carlen, Frank and Lieb. In particular, we confirm a weaker conjecture of Audenaert and Datta and obtain the full range of $(\alpha,z)$ for $\alpha$-$z$ R\'enyi relative entropies to be monotone under completely positive trace preserving maps. We also give simpler proofs of many known results, including the concavity of $\Psi_{p,0,1/p}$ for $0<p<1$ which was first proved by Epstein using complex analysis. The key is to reduce the problem to the joint convexity/concavity of the trace functions
\[
\Psi_{p,1-p,1}(A,B)=\Tr K^*A^{p}KB^{1-p},~~-1\le p\le 1,
\]
using a variational method. 
\end{abstract}

\section{Introduction}
The joint convexity/concavity of the trace functions
\begin{equation}\label{equ:defn of Psi_p,q,s}
\Psi_{p,q,s}(A,B)=\Tr(B^{\frac{q}{2}}K^*A^{p}KB^{\frac{q}{2}})^s,~~p,q,s\in \mathbb{R},
\end{equation}
has played an important role in mathematical physics and quantum information. Its study can be traced back to the celebrated \emph{Lieb's Concavity Theorem} \cite{Lieb73WYD}, which states that $\Psi_{p,q,1}$ is jointly concave for all $0\le p,q\le1,p+q\le 1$ and for all $K$. Using this, Lieb confirmed the Wigner-Yanase-Dyson conjecture \cite{WY63}: \emph{for $0<p<1$ and any self-adjoint $K$, the function 
\begin{equation}\label{eq:skew information}
S_p(\rho,K):=\frac{1}{2}\Tr[\rho^p,K][\rho^{1-p},K]=-\Tr\rho K^2+\Tr\rho^{p}K\rho^{1-p} K,
\end{equation}
is concave in $\rho$, where $[A,B]=AB-BA$.} We refer to \cite{WY63,Lieb73WYD} for more details about the \emph{skew information} $-S_p(\rho,K)$.
\smallskip

Since then, a lot of work around the joint convexity/concavity of $\Psi_{p,q,s}$ has emerged \cite{Ando79,Bekjan04convexity,CFL16some,CL08minkowski-II,CL99minkowski-I,Epstein73,FL13DPIsandwich,Hiai13concavity-I,Hiai16concavity-II}, following \cite{Lieb73WYD}. Through this line of research many methods have been developed. Two main methods are the ``analytic method'' and the ``variational method''. We refer to a very nice survey paper \cite{CFL18conjecture} for more historical information and the explanation of these two methods.

\smallskip
Another motivation to study the joint convexity/concavity of $\Psi_{p,q,s}$ comes from quantum information theory. Indeed, the joint convexity/concavity of $\Psi_{p,q,1/(p+q)}$ is closely related to the monotonicity (or \emph{Data Processing Inequality}) of the $\a$-$z$ R\'enyi relative entropies, which has become a frontier topic in recent years. We shall recall this in Section \ref{sect:backgroud of quantum information}. Starting from this Audenaert and Datta conjectured that:

\begin{con}\cite[Conjecture 1]{AD15alpha-z}\label{conj:Audenaert-Datta}
	If $1\leq p\leq 2,~-1\leq q<0$ and $(p,q)\ne(1,-1)$, then for any matrix $K$, the function
	\[
	\Psi_{p,q,1/(p+q)}(A,B)=\Tr(B^{\frac{q}{2}}K^*A^{p}KB^{\frac{q}{2}})^{\frac{1}{p+q}},
	\]
	is jointly convex in $(A,B)$, where $A$ and $B$ are positive definite matrices.
\end{con}
We cheat a little bit here, since the original form of their conjecture concerns the convexity of $ A\mapsto \Tr(A^{\frac{q}{2}}K^*A^{p}KA^{\frac{q}{2}})^{\frac{1}{p+q}}$ for all $K$. However, by doubling dimension, a standard argument shows that they are equivalent. See the discussions after \cite[Conjecture 1]{CFL18conjecture} for example.

In this paper we confirm a stronger conjecture of Carlen, Frank and Lieb:

\begin{con}\label{conj:CFL}\cite[Conjecture 4]{CFL18conjecture}
	If $1\leq p\leq 2,~-1\leq q<0,~(p,q)\ne(1,-1)$ and $s\ge \frac{1}{p+q}$, then for any matrix $K$, the function
	\begin{equation*}
	\Psi_{p,q,s}(A,B)=\Tr(B^{\frac{q}{2}}K^*A^{p}KB^{\frac{q}{2}})^s,	
	\end{equation*}
	is jointly convex in $(A,B)$, where $A$ and $B$ are positive definite matrices.
\end{con}

\begin{figure}[h]
	\centering
	\begin{tikzpicture}[node distance=2cm]
	\draw[->](-4,0)--(7,0) node[right]{p};
	\draw[->](0,-4)--(0,7) node[above]{q};
	
	\coordinate [label=225:$o$](A) at (0,0);
	\coordinate (B) at (0,3);
	\coordinate (C) at (3,3);
	\coordinate (D) at (3,0);
	\fill[lime] (A) -- (B) -- (C) -- (D) -- cycle;
	\draw (A)--(B)--(C)--(D)--cycle;
	\draw (1.5,50pt) node[] {\textbf{concave}};
	\draw (1.5,35pt) node[] {for $0\le s \le \frac{1}{p+q} $};
	
	\coordinate (E) at (6,0);
	\coordinate (F) at (6,-3);
	\coordinate (G) at (3,-3);
	\fill[yellow] (D) -- (E) -- (F) -- (G) -- cycle;
	\draw(D)--(E)--(F)--(G)--cycle;
	\draw (4.5,-35pt) node[] {\textbf{convex}};
	\draw (4.5,-50pt) node[] {for $s \ge \frac{1}{p+q} $};

	\coordinate (H) at (0,-3);
	\coordinate (I) at (-3,-3);
	\coordinate (J) at (-3,0);
	\fill[yellow] (A) -- (H) -- (I) -- (J) -- cycle;
	\draw(A)--(H)--(I)--(J)--cycle;
	\draw (-1.5,-35pt) node[] {\textbf{convex}};
	\draw (-1.5,-50pt) node[] {for $s \geq 0 $};
	
	\coordinate (K) at (-3,3);
	\coordinate (L) at (-3,6);
	\coordinate (M) at (0,6);
	\fill[yellow] (B) -- (K) -- (L) -- (M) -- cycle;
	\draw(B)--(K)--(L)--(M)--cycle;
	\draw (-1.5,135pt) node[] {\textbf{convex}};
	\draw (-1.5,120pt) node[] {for $s \ge \frac{1}{p+q} $};
	
	\draw (-3,0) node[above] {-1};
	\draw (0,-3) node[right] {-1};
	\draw (3.2,0) node[above] {1};
	\draw (6,0) node[above] {2};
	\draw (0,3.2) node[right] {1};
	\draw (0,6) node[right] {2};
	
	\draw[dashed] (-2.5,-2.5)--(-3.5,-3.5);
	\draw[dashed] (-0.5,-0.5)--(0.5,0.5);
	\draw[dashed] (2.5,2.5)--(6.5,6.5) node[above] {q=p};

	\end{tikzpicture}
	\caption{Joint convexity/concavity (for all $K$) of $\Psi_{p,q,s}$}\label{figure:convex/concave}
\end{figure}
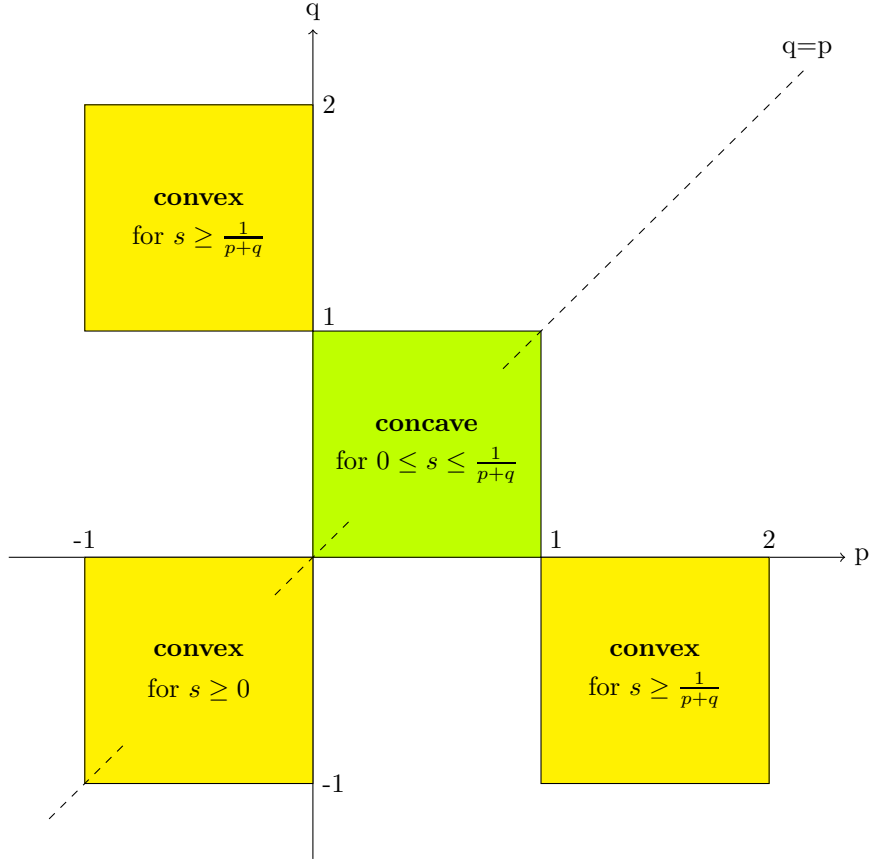

Consequently, we give the full range of $(p,q,s)$ for $\Psi_{p,q,s}$ to be jointly convex or jointly concave for any invertible $K$. See Figure \ref{figure:convex/concave} (note that $(1,-1)$ and $(-1,1)$ do not belong to the area of convexity) and the following

\begin{thm}\label{thm:A}
	Fix any invertible matrix $K$. Suppose that $p\geq q$ and $s>0$. Then $\Psi_{p,q,s}$ defined in \eqref{equ:defn of Psi_p,q,s} is
	\begin{enumerate}
		\item jointly concave if $0\leq q\leq p\leq1$ and $0<s\leq\frac{1}{p+q}$;
		\item jointly convex if $-1\leq q\leq p\leq 0$ and $s>0$;
		\item jointly convex if $-1\leq q\leq0,~1\leq p\leq 2,~(p,q)\ne(1,-1)$ and $s\geq\frac{1}{p+q}$.
	\end{enumerate}
\end{thm}

We remark here that the symmetric property of $\Psi_{p,q,s}$ allows us to assume $p\geq q$ and $s>0$. See the discussions before Proposition \ref{prop:known results on psi}. Moreover, the above result is sharp, in view of Proposition \ref{prop:ness for psi}.

\smallskip
As a corollary of Theorem \ref{thm:A}, Proposition \ref{prop:characterization of DPI} and Proposition \ref{prop:ness for psi}, we obtain all $(\a,z)$ such that $D_{\a,z}$ is monotone under completely positive trace preserving maps (or satisfies Data Processing Inequality, see \eqref{ineq:DPI} for the precise definition). 

\begin{thm}
	The $\a$-$z$ relative R\'enyi entropy $D_{\a,z}$ is monotone under completely positive trace preserving maps if and only if one of the following holds
	\begin{enumerate}
		\item $0<\a<1$ and $z\geq\max\{\a,1-\a\}$;
		\item $1<\a\leq 2$ and $\frac{\a}{2}\leq z\leq\a$;
		\item $2\leq\a<\infty$ and $\a-1\leq z\leq\a$.
	\end{enumerate}
\end{thm}

As we mentioned earlier, in the history two main methods have been developed to study the convexity/concavity of the trace functions $\Psi_{p,q,s}$: the analytic method and the variational method. The analytic method, which is the methodology employing the theory of Herglotz functions, was first introduced by Epstein \cite{Epstein73}. The variational method was first used by Carlen and Lieb in \cite{CL08minkowski-II}. Both of them have their own advantages, as the authors wrote in \cite[Page 8]{CFL18conjecture}: ``It appears that the analyticity method is especially useful for proving concavity and the variational method is more useful for proving convexity, but this is not meant to be an absolute distinction.'' In this paper we confirm Conjecture \ref{conj:CFL} by developing only the variational method.

\smallskip

The main value of this paper is twofold. Firstly, we develop the variational method in a very simple way such that it is useful to prove both convexity and concavity, and it reduces the convexity/concavity of $\Psi_{p,q,s}$ to three very particular cases, which were already known (see Theorem \ref{thm:reduction of convexity/concavity}). In this way we obtain the full range of $(p,q,s)$ such that $\Psi_{p,q,s}$ is jointly convex/concave and confirm Conjecture \ref{conj:Audenaert-Datta} and Conjecture \ref{conj:CFL}. Secondly, using our variational method in a slightly different way, we can furthermore reduce these three very particular cases to Lieb's concavity result \cite{Lieb73WYD} of $\Psi_{p,1-p,1}$ for $0< p\le 1$ and Ando's convexity result \cite{Ando79} of $\Psi_{p,1-p,1}$ for $-1\le p<0$. In other words, from Lieb's and Ando's classical joint convexity/concavity results (which admit many simple proofs) on 
$$\Psi_{p,1-p,1}(A,B)=\Tr K^*A^p K B^{1-p},$$
the subsequent results on joint convexity/concavity of $\Psi_{p,q,s}$ can be derived easily via our variational method. In this way we recover many classical results immediately. Moreover, we emphasize here that the analytic method can be avoided. 

In the past half a century we have developed a lot of tools to tackle the convexity/concavity of trace functions, and have witnessed a number of applications of the convexity/concavity of trace functions to many areas, like mathematical physics and quantum information. Now our variational method helps us to reduce the Carlen-Frank-Lieb conjecture (in fact the joint convexity/concavity of the whole family $\Psi_{p,q,s}$) to the convexity/concavity of the trace function \eqref{eq:skew information} (in which the essential part is $\Psi_{p,1-p,1}$) in the Wigner-Yanase-Dyson conjecture. This brings us back to the origin of the whole story.

\smallskip
This paper is organized as follows. In Section \ref{sect:backgroud of quantum information} we recall the background of Conjecture \ref{conj:Audenaert-Datta} and Conjecture \ref{conj:CFL}. In Section \ref{sect:proofs} we give the proof of our main result Theorem \ref{thm:A}.
\smallskip

We fix some notations in this paper. We use $\h$ to denote a finite-dimensional Hilbert space. We use $\bh$ to denote the family of bounded linear operators on $\h$, $\ph$ to denote the family of positive linear operators on $\h$ (or $n$-by-$n$ positive semi-definite matrices with $\dim\h=n$), and $\dh$ to denote the family of density operators, i.e., positive linear operators on $\h$ with unit trace (or $n$-by-$n$ positive semi-definite matrices having unit trace with $\dim\h=n$). Moreover, we use $\bhx$ (reps. $\phx$ and $\dhx$) to denote the family of invertible operators in $\bh$ (resp. $\ph$ and $\dh$). We use $\Tr$ to denote the usual trace on matrix algebra and we use $I$ to denote the identity matrix. For any matrix $A$ we use $|A|$ to denote its modulus $(A^*A)^{\frac{1}{2}}$.
\smallskip

We close this section with a remark. In this paper we are mainly dealing with the invertible matrices, to avoid some technical problems and make the paper more readable. In this case for $A\in \ph^{\times}$ and $\a\in\real$, $A^{\a}$ is always well-defined. Some results in this paper are still valid in the non-invertible case, by using an approximation argument. For example, in Conjectures \ref{conj:Audenaert-Datta} and \ref{conj:CFL}, $K$ is not assumed to be invertible, since $X^s$ is always well-defined for positive semi-definite $X$ and $s>0$. When $K$ is not invertible, one can approximate $K$ with invertible $K_\epsilon=K+\epsilon I$, where $\epsilon>0$ is small enough. Then the convexity of $\Tr(B^{\frac{q}{2}}K^*A^{p}KB^{\frac{q}{2}})^s$, which is the limit of $\Tr(B^{\frac{q}{2}}K_\epsilon^*A^{p}K_\epsilon B^{\frac{q}{2}})^s$ as $\epsilon$ tends to 0, follows from that of $\Tr(B^{\frac{q}{2}}K_\epsilon^*A^{p}K_\epsilon B^{\frac{q}{2}})^s$, since the convexity is stable under taking limits. 
%Note that we always assume an operator to be invertible. Most of the results in this paper hold also for non-invertible case, due to a standard approximation argument. If $A$ is non-invertible and $\a<0$, $A^{\a}$ is understood as $(A|_{\text{supp}(A)})^{\a}$, where $\text{supp}(A)$ denotes the support of $A$.

%, so does the conjecture of Carlen, Frank and Lieb \cite[Conjecture 4]{CFL18conjecture} (then the conjecture of Audenaert and Datta \cite[Conjecture 1]{AD15alpha-z}).

%One thing we want to emphasis here is that, the study of the joint convexity/concavity of $\Psi_{p,q,s}$ starts from the Wigner-Yanase-Dyson Conjecture, since it concerns the concavity of $\Psi_{p,1-p,1}$. But more general functions $\Psi_{p,q,1/(p+q)}$ have important applications in quantum information theory. In fact, it is closely related to the monotonicity (Data Processing Inequality) of $\a$-$z$ R\'enyi relative entropies,  which has become a frontier topic in recent years. This is the motivation of the conjecture of Audenaert and Datta we mentioned earlier. So we will start with the background in quantum information. However, we shall see that the joint convexity/concavity of the most general functions $\Psi_{p,q,s}$ can be easily reduced to that of functions $\Psi_{p,1-p,1}$, where the whole story begins with.

\section{Background}\label{sect:backgroud of quantum information}
In this section we collect necessary background information for this paper. Most of them are borrowed from the survey paper \cite{CFL18conjecture}. One can refer to \cite{CFL18conjecture} and the references therein for further details. Experts may skip this section without any difficulty.

\smallskip
Given two probability density functions $P$ and $Q$ on $\mathbb{R}$, the \emph{relative entropy}, or \emph{Kullback-Leibler divergence} of $P$ with respect to $Q$ is given by 
\begin{equation}\label{equ:classical relative entropy}
S(P||Q):=\int_{\mathbb{R}}P(x)(\log P(x)-\log Q(x))dx.
\end{equation}
For $\a\in(0,1)\cup(1,\infty)$, the \emph{$\a$-R\'enyi relative entropy} of $P$ with respect to $Q$ is defined as \cite{Renyi61}
\begin{equation}\label{equ:classcial alpha relative entropy}
S_{\a}(P||Q):=\frac{1}{\a-1}\log\int_{\mathbb{R}}P(x)^{\a}Q(x)^{1-\a}dx.
\end{equation}

Both classical relative entropies \eqref{equ:classical relative entropy} and \eqref{equ:classcial alpha relative entropy} have been generalized to quantum setting, where the density functions are replaced by the density operators, and the integral is replaced by the trace, respectively. However, their quantum analogues might take various forms.
\smallskip

Fix $\rho,\sigma\in\dh^{\times}$ with $\h$ being any finite-dimensional Hilbert space. A natural quantum analogue of \eqref{equ:classical relative entropy}, is the so-called \emph{Umegaki relative entropy} \cite{Umegaki62}
\begin{equation}\label{equ:Umegaki relative entropy}
D(\rho||\sigma):=\Tr\rho(\log\rho-\log\sigma).
\end{equation}
It is monotone under completely positive trace preserving (CPTP) maps \cite{Lindblad75}. That is,
\begin{equation}\label{ineq:DPI for D}
D(\mathcal{E}(\rho)||\mathcal{E}(\sigma))\leq D(\rho||\sigma),
\end{equation}
for all CPTP maps $\mathcal{E}:\bh\to\bh$ and all density operators $\rho,\sigma\in\dhx$. 
%Recall that a \emph{quantum channel} is a completely positive trace preserving map $\mathcal{E}:\bh\to\mathcal{B(H')}$, where $\h$ and $\h'$ are two finite dimensional Hilbert spaces.

The inequality \eqref{ineq:DPI for D} is known as the \emph{Data Processing Inequality} (DPI). As one of the most fundamental inequalities in quantum information, DPI has strong links with the Strong Subadditivity (SSA) of the von Neumann entropy \cite{LR73SSA}, the uncertainty principle \cite{TR2011}, the quantum hypothesis testing \cite{MO15hypothesis} and the Holevo bound for the accessible information \cite{Holevo73}. Not every quantum analogue of \eqref{equ:classical relative entropy} satisfies DPI. For example, it is known that \cite{CL18some}
\begin{equation*}
D'(\rho||\sigma):=\Tr\rho\log(\sigma^{-\frac{1}{2}}\rho\sigma^{-\frac{1}{2}}),
\end{equation*}
as a generalization of \eqref{equ:classical relative entropy}, does not satisfy DPI.

A natural generalization of \eqref{equ:classcial alpha relative entropy} is the family of \emph{quantum $\a$-R\'enyi relative entropies}
\begin{equation*}
D_{\a}(\rho||\sigma):=\frac{1}{\a-1}\log\Tr(\rho^{\a}\sigma^{1-\a}),~~ \a\in(0,1)\cup(1,\infty).
\end{equation*}
%It admits Umegaki relative entropy $D(\rho||\sigma)$ as a limit case when $\a\to 1$. 

Another important generalization of \eqref{equ:classcial alpha relative entropy}, introduced by M\"uller-Lennert, Dupuis, Szehr, Fehr, Tomamichel \cite{MDSFT13sandwich} and Wilde, Winter, Yang \cite{WWY14sandwich}, are the \emph{sandwiched $\a$-R\'enyi entropies}:
\begin{equation*}
\widetilde{D}_{\a}(\rho||\sigma):=\frac{1}{\a-1}\log\Tr(\sigma^{\frac{1-\a}{2\a}}\rho\sigma^{\frac{1-\a}{2\a}})^{\a},~~\a\in(0,1)\cup(1,\infty).
\end{equation*}

Audenaert and Datta \cite{AD15alpha-z} introduced a new family of quantum R\'enyi relative entropies by using two parameters, called the \emph{$\a$-$z$ R\'enyi relative entropies}:
\begin{equation}\label{equ:alpha-z entropy}
D_{\a,z}(\rho||\sigma):=\frac{1}{\a-1}\log\Tr(\sigma^{\frac{1-\a}{2z}}\rho^{\frac{\a}{z}}\sigma^{\frac{1-\a}{2z}})^{z},~~\a\in(-\infty,1)\cup(1,\infty),~~z>0.
\end{equation} 
It unifies $D_{\a}$ and $\widetilde{D}_{\a}$ by taking $z=1$ and $z=\a$, respectively. We comment here that the $\a$-$z$ R\'enyi relative entropies have appeared earlier in a paper by Jaksic, Ogata, Pautrat and Pillet \cite{JOPP12entropic}.

\smallskip
\begin{comment}

We recommend here a very nice paper \cite{CFL18conjecture} by Carlen, Frank and Lieb, which tells the story of relative entropies more seamlessly and enables new comers to this field, including the author, to understand the background and get access to recent advances more easily and quickly. Most of notions in this paper come from \cite{CFL18conjecture} and one can find all needed information there and the references therein.

\smallskip
\end{comment}

A natural question is, for which $(\a,z)$ does the $\a$-$z$ R\'enyi relative entropy $D_{\a,z}$ satisfy DPI, that is,
\begin{equation}\label{ineq:DPI}
D_{\a,z}(\mathcal{E}(\rho)||\mathcal{E}(\sigma))\leq D_{\a,z}(\rho||\sigma),
\end{equation}
for any CPTP map $\mathcal{E}$ on $\bh$ and all density operators $\rho,\sigma\in\dh^{\times}$? This remained open for some range of $(\a,z)$ before the present paper. It is well-known that DPI is essentially equivalent to the joint convexity/concavity of the trace functions inside the definition of $D_{\a,z}$.

\begin{prop}\label{prop:characterization of DPI}\cite[Proposition 7]{CFL18conjecture}
	Let $\a\in(-\infty,1)\cup(1,\infty)$ and $z>0$. Set $p=\frac{\a}{z}$ and $q=\frac{1-\a}{z}$. Then \eqref{ineq:DPI} holds for any CPTP map $\mathcal{E}:\bh\to \bh$, all density operators $\rho,\sigma\in\dh^{\times}$ and any finite-dimensional Hilbert space $\h$ if and only if one of the following holds
	\begin{enumerate}
		\item $\a<1$ and $\Psi_{p,q,1/(p+q)}$ with $K=I$ is jointly concave;
		\item $\a>1$ and $\Psi_{p,q,1/(p+q)}$ with $K=I$ is jointly convex.
	\end{enumerate}
\end{prop}

For the reader's convenience, we present its proof in the end of this section. From some known results on the joint convexity/concavity of $\Psi_{p,q,1/(p+q)}$ with $K=I$, Audenaert and Datta obtained DPI for $D_{\a,z}$ for some---but not full---range of $(\a,z)$ \cite[Theorem 1]{AD15alpha-z}. By saying full we mean necessary and sufficient conditions on $(\a,z)$. It is then natural to ask whether DPI holds for the remaining range of $(\a,z)$. This motivated Audenaert and Datta to raise Conjecture \ref{conj:Audenaert-Datta}.

\smallskip

More generally, consider the joint convexity/concavity of trace functions
\[
\Psi_{p,q,s}(A,B)=\Tr(B^{\frac{q}{2}}K^*A^{p}KB^{\frac{q}{2}})^s,
\]
where $A,B\in \phx$, $K\in \bhx$ and $p,q,s\in\mathbb{R}$. Note that $\Psi_{q,p,s}(B,A)=\Psi_{p,q,s}(A,B)$ with $K$ replaced by $K^*$, and $\Psi_{-p,-q,-s}(A,B)=\Psi_{p,q,s}(A,B)$ with $K$ replaced by $(K^{-1})^*$. So in the sequel we assume that $p\geq q$ and $s>0$.

The knowledge of the joint convexity/concavity of $\Psi_{p,q,s}$ before the survey paper \cite{CFL18conjecture} is summarized in the following proposition  in \cite{CFL18conjecture} or the figure therein.

\begin{prop}\label{prop:known results on psi}\cite[Theorem 2]{CFL18conjecture}
	Fix $K\in \bhx$. Then $\Psi_{p,q,s}$ is
	\begin{enumerate}
		\item jointly concave if $0\leq q\leq p\leq1$ and $0<s\leq\frac{1}{p+q}$;
		\item jointly convex if $-1\leq q\leq p\leq 0$ and $s>0$;
		\item jointly convex if $-1\leq q\leq0,~1\leq p<2,~(p,q)\ne(1,-1)$ and $s\geq\min\{\frac{1}{p-1},\frac{1}{q+1}\}$ or $p=2,~-1\leq q\leq0$ and $s\geq\frac{1}{q+2}$.
	\end{enumerate}
\end{prop}

%The proofs of (1)(2) for full range are due to Hiai \cite[Theorem 2.1]{Hiai16concavity-II}. The proofs of (3) are due to Frank and Lieb \cite[Proposition 3]{FL13DPIsandwich}, and Carlen, Frank and Lieb \cite{CFL16some}. 
For more historical details of these results, see the discussions after \cite[Theorem 2]{CFL18conjecture}. We only comment here that the case $s=1$, which was first studied in the history, is due to Lieb \cite{Lieb73WYD} for $0\leq q\leq p\leq1$ with $p+q\leq 1$, as well as for $-1\leq q\le p\leq0$, and due to Ando \cite{Ando79} for $-1\leq q\leq0,~1\leq p<2$, with $p+q\geq1$. Their work played an important role in the development of matrix analysis.

The following proposition, due to Hiai \cite{Hiai13concavity-I}, gives the necessary conditions for $\Psi_{p,q,s}$ to be jointly convex or jointly concave. 

\begin{prop}\cite[Propositions 5.1(2) and 5.4(2)]{Hiai13concavity-I}\cite[Proposition 3]{CFL18conjecture}\label{prop:ness for psi}
	Let $p\geq q$ and $s>0$. Suppose that $(p,q)\ne(0,0)$ and $K=I$.
	\begin{enumerate}
		\item If $\Psi_{p,q,s}$ is jointly concave for $\h=\mathbb{C}^2$, then $0\leq q\leq p\leq1$ and $0<s\leq\frac{1}{p+q}$.
		\item If $\Psi_{p,q,s}$ is jointly convex for $\h=\mathbb{C}^4$, then either $-1\leq q\leq p\leq 0$ and $s>0$ or $-1\leq q\leq0,~1\leq p\leq 2,~ (p,q)\ne(1,-1)$ and $s\geq\frac{1}{p+q}$.
	\end{enumerate}
\end{prop}

From the above two propositions, Carlen, Frank and Lieb raised Conjecture \ref{conj:CFL}. Some partial results were known before the present paper, as pointed out in Proposition \ref{prop:known results on psi} (3). 

\smallskip

We close this section with the proof of Proposition \ref{prop:characterization of DPI}. It comes from \cite[Proposition 7]{CFL18conjecture}, following a well-known argument due to Lindblad \cite{Lindblad75} and Uhlmann \cite{Uhlmann73}. 

\begin{proof}[Proof of Proposition \ref{prop:characterization of DPI}]
	We use $\Psi$ to denote $\Psi_{p,q,1/(p+q)}$ with $K=I$. We only prove the case $\a>1$, since the proof for $\a<1$ is similar. Then it is equivalent to show that $\Psi$ satisfies the inequality
	$$
	\Psi(\mathcal{E}(\rho),\mathcal{E}(\sigma))\le\Psi(\rho,\sigma),
	$$ 
	for any CPTP map $\mathcal{E}$ on $\bh$, for all $\rho,\sigma\in \dh^{\times}$ and for all $\h$ if and only if $\Psi$ is jointly convex.
	\smallskip
	
	To show the ``if'' part, take any CPTP map $\mathcal{E}:\bh\to \bh$. Then we can write $\mathcal{E}$ as
	\[
	\mathcal{E}(\gamma)=\Tr_2U(\gamma\otimes \delta)U^*,
	\]
	where $\delta\in\mathcal{D(H')}$, $U$ is unitary on $\h\otimes\h'$, and $\h'$ is a Hilbert space such that $N':=\dim \h'\le (\dim\h)^2$. Here $\Tr_2$ denotes the usual partial trace over $\h'$. For a proof, see for example \cite[Lemma 5]{Lindblad75}. It origins in the celebrated Stinespring's Theorem \cite{Stinespring55}. Let $du$ denote the normalized Haar measure on the group of all unitaries on $\h'$, then
	\begin{equation}\label{equ:Schur's lemma}
		\mathcal{E}(\gamma)\otimes\frac{I_{\h'}}{N'}=\int(I_{\h}\otimes u)U(\gamma\otimes\delta)U^*(I_{\h}\otimes u^*)du,
	\end{equation}
	where $I_{\h}$ and $I_{\h'}$ are the identity maps over $\h$ and $\h'$, respectively. By the tensor property of $\Psi$, we have
	\[
\Psi(\mathcal{E}(\rho),\mathcal{E}(\sigma))
=\Psi\left(\mathcal{E}(\rho)\otimes\frac{I_{\h'}}{N'},\mathcal{E}(\sigma)\otimes\frac{I_{\h'}}{N'}\right).
\]
	From the joint convexity of $\Psi$ and \eqref{equ:Schur's lemma} it follows that
	\[
\Psi(\mathcal{E}(\rho),\mathcal{E}(\sigma))
\leq\int\Psi((I_{\h}\otimes u)U(\rho\otimes\delta)U^*(I_{\h}\otimes u^*),(I_{\h}\otimes u)U(\sigma\otimes\delta)U^*(I_{\h}\otimes u^*))du.
\]
	By the unitary invariance and the tensor property of $\Psi$ we obtain that
	\[
	\Psi(\mathcal{E}(\rho),\mathcal{E}(\sigma))\leq\Psi(\rho,\sigma),
	\] 
	as desired.
	
	\smallskip
	
	To show the ``only if'' part, for any $\rho_1,\rho_2,\sigma_1,\sigma_2\in\dh^{\times}$ and any $0<\lambda<1$, define 
	\[
	\rho=
	\begin{pmatrix*}
	\lambda\rho_1&0\\
	0&(1-\lambda)\rho_2
	\end{pmatrix*}
	\text{ and }
	\sigma=
	\begin{pmatrix*}
	\lambda\sigma_1&0\\
	0&(1-\lambda)\sigma_2
	\end{pmatrix*},
	\]
	in $\mathcal{D}(\h\oplus\h)^{\times}$. Since the map 
	\begin{equation}\label{equ:CPTP map}
	\mathcal{E}
	\begin{pmatrix*}
	a&b\\
	c&d
	\end{pmatrix*}
	=\frac{1}{2}	\begin{pmatrix*}
	a+d&0\\
	0&a+d
	\end{pmatrix*},
	\end{equation}
	is a CPTP map, we obtain from the monotonicity of $\Psi$ that
	\[
	\Psi(\mathcal{E}(\rho),\mathcal{E}(\sigma))\leq\Psi(\rho,\sigma),
	\]
	which is nothing but
	\[
	\Psi(\lambda\rho_1+(1-\lambda)\rho_2,\lambda\sigma_1+(1-\lambda)\sigma_2)\leq\lambda\Psi(\rho_1,\sigma_1)+(1-\lambda)\Psi(\rho_2,\sigma_2).
	\]
	This finishes the proof of the joint convexity of $\Psi$.
\end{proof}

\section{The proofs}\label{sect:proofs}
%and the assumption of invertability will bring some restrictions. That is why we did not choose $\mathcal{E}$ in \eqref{equ:CPTP map} to be 
%	\begin{equation*}
%\mathcal{E}
%\begin{pmatrix*}
%a&b\\
%c&d
%\end{pmatrix*}
%=\begin{pmatrix*}
%a+d&0\\
%0&0
%\end{pmatrix*},
%\end{equation*}
%which people usually use. So let us remark here that, the non-invertible case is usually dealt with by an approximation argument.

%Now let us proceed with the proof.
This section is devoted to the proof of Theorem \ref{thm:A}. The following classical results will serve as the building blocks to achieve the joint convexity/concavity of $\Psi_{p,q,s}$. The concavity result is due to Lieb \cite{Lieb73WYD} and the convexity result is due to Ando \cite{Ando79}. They have now many simple proofs, see for example \cite{NEE13}. We only comment here that they are based on the operator convexity of $A\mapsto A^p$ when $-1\le p< 0$ or $1\le p\le 2$, and the operator concavity of $A\mapsto A^p$ when $0<p\le 1$.

\begin{lem}\cite{Lieb73WYD,Ando79}\label{lem:Lieb-Ando}
	For any $K\in \bh^{\times}$, the function
	$$\Psi_{p,1-p,1}(A,B)=\Tr K^*A^p KB^{1-p},~~A,B\in\ph^{\times},$$ 
	is 
	\begin{enumerate}
		\item jointly concave if $0<p\le1$;
		\item jointly convex if $-1\le p<0$.
	\end{enumerate}
\end{lem}

Theorem \ref{thm:A} will be reduced to Lemma \ref{lem:Lieb-Ando} in three steps, using a variational method. The idea of the variational method is based on the following lemma \cite[Lemma 13]{CFL18conjecture}. We give the proof here for the reader's convenience.

\begin{lem}\label{lem:convex-concave}
	Let $X,Y$ be two convex subsets of vector spaces and $f:X\times Y\to \real$ a function.
	\begin{enumerate}
		\item [(1)] If $f(\cdot,y)$ is convex (resp. concave) for any $y\in Y$, then $x \mapsto \sup_{y\in Y} f(x,y)$ (resp. $x \mapsto \inf_{y\in Y} f(x,y)$) is convex (resp. concave).
		\item [(2)] If $f$ is jointly convex (resp. concave) on $X\times Y$, then $x \mapsto \inf_{y\in Y} f(x,y)$ (resp. $x \mapsto \sup_{y\in Y} f(x,y)$) is convex (resp. concave).
	\end{enumerate}

\end{lem}

\begin{proof}
	\begin{enumerate}
		\item [(1)] This follows immediately from the definition.
		\item [(2)] We only prove the convexity here. The proof of the concavity is similar. For any $x_1,x_2\in X$ and any $0<\lambda<1$, set $x:=\lambda x_1+(1-\lambda)x_2$. Then for any $\epsilon>0$ and $i=1,2$, there exists $y_i\in Y$ such that $f(x_i,y_i)\le \inf_{y\in Y}f(x_i,y)+\epsilon$. By the joint convexity of $f$, we have 
		\begin{equation*}
		\begin{split}
		\inf_{y\in Y} f(x,y)
		&\le f(x,\lambda y_1+(1-\lambda)y_2)\\
		&\le \lambda f(x_1,y_1)+(1-\lambda)f(x_2,y_2)\\
		&\le \lambda \inf_{y\in Y}f(x_1,y)+ (1-\lambda) \inf_{y\in Y}f(x_2,y)+\epsilon.
		\end{split}
		\end{equation*}
		Then the proof finishes by letting $\epsilon\to 0^+$.
	\end{enumerate}
\end{proof}

The following variational method is the key of the proof. It originates in \cite{CL08minkowski-II} and the special cases (either $r_0=1$ or $r_1=1$) have been widely used \cite{CFL18conjecture}.

\begin{thm}
	For $r_i>0,i=0,1,2$ such that $\frac{1}{r_0}=\frac{1}{r_1}+\frac{1}{r_2}$, we have for any $X,Y\in\bh^{\times}$ that
	\begin{equation}\label{equ:variational method min-concave}
	\Tr|XY|^{r_0}=\min_{Z\in\bh^{\times}}\left\{\frac{r_0}{r_1}\Tr|XZ|^{r_1}+\frac{r_0}{r_2}\Tr|Z^{-1}Y|^{r_2}\right\},
	\end{equation}
	and 
	\begin{equation}\label{equ:variational method max-convex}
	\Tr|XY|^{r_1}=\max_{Z\in\bh^{\times}}\left\{\frac{r_1}{r_0}\Tr|XZ|^{r_0}-\frac{r_1}{r_2}\Tr|Y^{-1}Z|^{r_2}\right\}.
	\end{equation}
%	If $X$ and $Y$ are not invertible, then min/max should be replaced by inf/sup, respectively.
\end{thm}

\begin{proof}
	For any $p>0$ define $\Vert\cdot\Vert_p$ as $\|A\|^p_p:=\Tr|A|^p$. For any $Z\in\bh^{\times}$, we have by H\"older's inequality that
	\[
	\Tr|XY|^{r_0}\leq\Vert XZ\Vert^{r_0}_{r_1}\Vert Z^{-1}Y\Vert^{r_0}_{r_2}=[\Tr|XZ|^{r_1}]^{\frac{r_0}{r_1}}[\Tr|Z^{-1}Y|^{r_2}]^{\frac{r_0}{r_2}}.
	\]
	For a proof of H\"older's inequality, see \cite[Exercise IV.2.7]{Bhatia97matrixanalysis}. Actually it is a special case of \cite[Exercise IV.2.7]{Bhatia97matrixanalysis} by choosing the unitarily invariant norm $\vertiii{\cdot}$ to be $\|\cdot\|_{1}$. And \cite[Exercise IV.2.7]{Bhatia97matrixanalysis} can be proved by almost the same argument as the proof of \cite[Corollary IV.2.6]{Bhatia97matrixanalysis}, since \cite[Theorem IV.2.5]{Bhatia97matrixanalysis} is valid for all $r>0$.
	
	Then from Young's inequality for numbers (or AM-GM inequality): $x^{\a}y^{\beta}\leq\a x+\beta y$ for positive $x,y$ and positive $\a,\beta$ such that $\a+\beta=1$, it follows that
	\begin{equation}\label{ineq: Holder plus Young-inf}
	\Tr|XY|^{r_0}\leq[\Tr|XZ|^{r_1}]^{\frac{r_0}{r_1}}[\Tr|Z^{-1}Y|^{r_2}]^{\frac{r_0}{r_2}}\leq\frac{r_0}{r_1}\Tr|XZ|^{r_1}+\frac{r_0}{r_2}\Tr|Z^{-1}Y|^{r_2}.
	\end{equation}
	By exchanging $Y$ and $Z$, we have
	\begin{equation}\label{ineq: Holder plus Young-sup}
	\Tr|XY|^{r_1}\geq\frac{r_1}{r_0}\Tr|XZ|^{r_0}-\frac{r_1}{r_2}\Tr|Y^{-1}Z|^{r_2}.
	\end{equation}
	
	In view of \eqref{ineq: Holder plus Young-inf}, to prove \eqref{equ:variational method min-concave} it suffices to find a minimizer. For this let $Y^*X^*=U|Y^*X^*|$ be the polar decomposition of $Y^*X^*$, then $XYU=|Y^*X^*|$. Set $Z:=YU|Y^*X^*|^{-\frac{r_1}{r_1+r_2}}$, then we have 
	\[
	XZ=XYU|Y^*X^*|^{-\frac{r_1}{r_1+r_2}}=|Y^*X^*|^{\frac{r_2}{r_1+r_2}},~~Z^{-1}Y=|Y^*X^*|^{\frac{r_1}{r_1+r_2}}U^*.
	\]
	Using the facts that $\Vert\cdot \Vert_p$ is unitarily invariant and $\Vert A \Vert_p=\Vert A^*\Vert_p$ for all $A$, we have
	\[
	\Tr|XZ|^{r_1}=\Tr|Y^*X^*|^{\frac{r_1 r_2}{r_1+r_2}}=\Tr|XY|^{\frac{r_1 r_2}{r_1+r_2}}=\Tr|XY|^{r_0},
	\]
	and
	\[
	\Tr|Z^{-1}Y|^{r_2}=\Tr|Y^*X^*|^{\frac{r_1 r_2}{r_1+r_2}}=\Tr|XY|^{\frac{r_1 r_2}{r_1+r_2}}=\Tr|XY|^{r_0}.
	\]
	Hence $\Tr|XY|^{r_0}=\frac{r_0}{r_1}\Tr|XZ|^{r_1}+\frac{r_0}{r_2}\Tr|Z^{-1}Y|^{r_2}$, which proves \eqref{equ:variational method min-concave}.
	% For non-invertible $X$ and $Y$, $X_{\epsilon}:=X+\epsilon I$ and $Y_{\epsilon}:=Y+\epsilon I$ are both invertible for small $\epsilon>0$. Then we have shown that there exists invertible $Z_{\epsilon}$ such that 
	%\[
	%\Tr|X_{\epsilon}Y_{\epsilon}|^{r_0}=\frac{r_0}{r_1}\Tr|XZ_{\epsilon}|^{r_1}+\frac{r_0}{r_2}\Tr|Z_{\epsilon}^{-1}Y|^{r_2}.
	%\]
	%Thus the proof of \eqref{equ:variational method min-concave} is finished, as soon as one observes $\Tr|XY|^{r_0}=\lim\limits_{\epsilon\to 0^+}\Tr|X_{\epsilon}Y_{\epsilon}|^{r_0}$.
	
	In view of \eqref{ineq: Holder plus Young-sup}, to prove \eqref{equ:variational method max-convex} it suffices to find a maximizer. For this let $U$ be as above and choose $Z$ to be $YU|Y^*X^*|^{\frac{r_1}{r_2}}$, then
	\[
	XZ=XYU|Y^*X^*|^{\frac{r_1}{r_2}}=|Y^*X^*|^{\frac{r_1+r_2}{r_2}},~~
	Y^{-1}Z=U|Y^*X^*|^{\frac{r_1}{r_2}}.
	\]
	It follows that 
	\[
	\Tr|XZ|^{r_0}=\Tr|Y^*X^*|^{\frac{(r_1+r_2)r_0}{r_2}}=\Tr|Y^*X^*|^{r_1}=\Tr|XY|^{r_1},
	\]
	and 
	\[    
	\Tr|Y^{-1}Z|^{r_2}=\Tr|Y^*X^*|^{r_1}=\Tr|XY|^{r_1}.
	\] 
	Hence $\Tr|XY|^{r_1}=\frac{r_1}{r_0}\Tr|XZ|^{r_0}-\frac{r_1}{r_2}\Tr|Y^{-1}Z|^{r_2}$ and the proof of \eqref{equ:variational method max-convex} is finished.
\end{proof}

\begin{rem}
	It is possible to generalize this variational method to the infinite dimensional case or to more general norm functions, which is beyond the aim of this paper. It is also possible to apply this variational method to trace functions with $n\geq 3$ variables. Let $r_j>0,j=0,1,\dots,n$ such that $\frac{1}{r_0}=\sum_{j=1}^{n}\frac{1}{r_j}$. Then we have for $X_1,\dots,X_n\in\bh^{\times}$ that
	\begin{equation}\label{equ:n-variables-min}
	\begin{split}
	&\Tr|X_1\cdots X_n|^{r_0}\\
	=&\min\left\{\frac{r_0}{r_1}\Tr|X_1Z_1|^{r_1}+\sum_{j=2}^{n-1}\frac{r_0}{r_j}\Tr|Z^{-1}_{j-1}X_jZ_j|^{r_j}+\frac{r_0}{r_n}\Tr|Z^{-1}_{n-1}X_n|^{r_n}\right\},
	\end{split}
	\end{equation}
	and 
	\begin{equation}\label{equ:n-variable-max}
	\begin{split}
	&\Tr|X_1\cdots X_n|^{r_1}\\
	=&\max\left\{\frac{r_1}{r_0}\Tr|X_{1}Z_{1}|^{r_0}-\sum_{j=2}^{n-1}\frac{r_1}{r_{j}}\Tr|Z^{-1}_{j}X^{-1}_{j}Z_{j-1}|^{r_{j}}-\frac{r_1}{r_n}\Tr|X^{-1}_{n}Z_{n-1}|^{r_n}\right\},
	\end{split}
	\end{equation}
	where min and max run over all $Z_1,\dots,Z_{n-1}\in \bh^{\times}$. The proof is similar to the two variables case. We only explain here that min is indeed achieved for \eqref{equ:n-variables-min}. Let $X^*_n\cdots X^*_1=U|X^*_n\cdots X^*_1|$ be the polar decomposition of $X^*_n\cdots X^*_1$.  Then set
	\[
	Z_j:=X_{j+1}\cdots X_nU|X^*_n\cdots X^*_1|^{\a_j},~~\a_j=\sum_{k=1}^{j}\frac{r_0}{r_{k}}-1
	\]
	for $1\leq j\leq n-1$. One can check that 
	\[
	\Tr|X_1\cdots X_n|^{r_0}=\frac{r_0}{r_1}\Tr|X_1Z_1|^{r_1}+\sum_{j=2}^{n-1}\frac{r_0}{r_j}\Tr|Z^{-1}_{j-1}X_jZ_j|^{r_j}+\frac{r_0}{r_n}\Tr|Z^{-1}_{n-1}X_n|^{r_n}.
	\]
\end{rem}

Now we are ready to proceed with the three steps of reductions. Note that \textbf{Step 1} is enough to finish the proof of Theorem \ref{thm:A} and confirm Conjectures \ref{conj:Audenaert-Datta} and \ref{conj:CFL}.

\smallskip
\textbf{Step 1:} In the first step we reduce the joint convexity/concavity of $\Psi_{p,q,s}$ to the convexity/concavity of
\[
\Upsilon_{p,s}(A):=\Tr(K^*A^{p}K)^s,~~A\in\ph^{\times},
\]
for all $K\in \bhx$, which has already been thoroughly studied.

\begin{thm}\cite[Proposition 5]{CFL18conjecture}\label{thm:one variable case}
	For any $K\in \bhx$, $\Upsilon_{p,s}$ is
	\begin{enumerate}
		\item concave if $0< p\leq1$ and $0<s\leq\frac{1}{p}$;
		\item convex if $-1\leq p\leq 0$ and $s>0$;
		\item convex if $1\leq p\leq 2$ and $s\geq\frac{1}{p}$.
	\end{enumerate}
\end{thm}

%The proofs of (1) and (2) are due to Hiai \cite[Theorem 4.1]{Hiai13concavity-I}. The proof of (3) is due to Carlen and Lieb \cite[Theorem 1.1]{CL08minkowski-II}. 
See the discussions after Proposition 5 in \cite{CFL18conjecture} for more historical information. We only comment here that the proof of concavity for $0< p\leq 1$ with $s=\frac{1}{p}$ is due to Epstein \cite{Epstein73}. His analytic method is nowadays developed as an important tool in matrix analysis, in particular to deal with concavity (rather than convexity) of trace functions. We will give a simpler proof of this theorem later, without using Epstein's analytic approach.

\smallskip

%Now we are ready to prove Theorem \ref{thm:A}.

\begin{proof}[Proof of Theorem \ref{thm:A} given Theorem \ref{thm:one variable case}]
	Before proceeding with the proof note first that 
	\[
	\Psi_{p,q,s}(A,B)=\Tr(B^{\frac{q}{2}}K^*A^{p}KB^{\frac{q}{2}})^s=\Tr|A^{\frac{p}{2}}KB^{\frac{q}{2}}|^{2s}.
	\]
	\smallskip
	
	(1) If $q=0$, then the claim reduces to Theorem \ref{thm:one variable case} (1). To show the case $0<q\leq p\leq 1$ and $0< s\leq\frac{1}{p+q}$, set $\lambda:=s(p+q)\in(0,1]$ and we apply \eqref{equ:variational method min-concave} to $(r_0,r_1,r_2)=(2s,\frac{2\lambda}{p},\frac{2\lambda}{q})$ and $(X,Y)=(A^{\frac{p}{2}}K,B^{\frac{q}{2}})$:
	\begin{equation}\label{equ:proof of thm psi 1}
	\Psi_{p,q,s}(A,B)=\min_{Z\in\bh^{\times}}\left\{\frac{p}{p+q}\Tr|A^{\frac{p}{2}}KZ|^{\frac{2\lambda}{p}}+\frac{q}{p+q}\Tr|Z^{-1}B^{\frac{q}{2}}|^{\frac{2\lambda}{q}} \right\}.
	\end{equation}
	Since $0<\frac{\lambda}{p}\leq\frac{1}{p}$ and $0<\frac{\lambda}{q}\leq\frac{1}{q}$, from Theorem \ref{thm:one variable case} (1) it follows that the maps
	\[
	A\mapsto\frac{p}{p+q}\Tr|A^{\frac{p}{2}}KZ|^{\frac{2\lambda}{p}}=\frac{p}{p+q}\Tr(Z^*K^*A^{p}KZ)^{\frac{\lambda}{p}}
	\]
	and 
	\[
	B\mapsto\frac{q}{p+q}\Tr|Z^{-1}B^{\frac{q}{2}}|^{\frac{2\lambda}{q}}=\frac{q}{p+q}\Tr(Z^{-1}B^{q}(Z^{-1})^*)^{\frac{\lambda}{q}}
	\]
	are both concave. Hence they are both jointly concave as functions in $(A,B)$ and so is $\Psi_{p,q,s}$ by Lemma \ref{lem:convex-concave} (1) and \eqref{equ:proof of thm psi 1}.
	\smallskip
	
	(2) If $p=0$, then the claim reduces to Theorem \ref{thm:one variable case} (2). Suppose $-1\leq q\leq p<0$ and $s>0$, then we apply \eqref{equ:variational method max-convex} to $(r_0,r_1,r_2)=(2t,2s,\frac{2}{-q})$ with $\frac{1}{t}=\frac{1}{s}-q$ and $(X,Y)=(A^{\frac{p}{2}}K,B^{\frac{q}{2}})$:
	\begin{equation}\label{equ:proof of thm psi 2}
	\Psi_{p,q,s}(A,B)=\max_{Z\in\bh^{\times}}\left\{\frac{s}{t}\Tr|A^{\frac{p}{2}}KZ|^{2t}+sq\Tr|B^{-\frac{q}{2}}Z|^{\frac{2}{-q}}\right\}.
	\end{equation}
	Note that $t>0$, $sq<0$ and $0<-q\leq 1$. By Theorem \ref{thm:one variable case} (1) and (2), the maps
	\[
	A\mapsto\frac{s}{t}\Tr|A^{\frac{p}{2}}KZ|^{2t}=\frac{s}{t}\Tr(Z^*K^*A^{p}KZ)^{t}
	\]
	and 
	\[
	B\mapsto sq\Tr|B^{-\frac{q}{2}}Z|^{\frac{2}{-q}}=sq\Tr(Z^*B^{-q}Z)^{\frac{1}{-q}}
	\]
	are both convex. Hence they are both jointly convex as functions in $(A,B)$ and so is $\Psi_{p,q,s}$ by Lemma \ref{lem:convex-concave} (1) and \eqref{equ:proof of thm psi 2}.
	
	\smallskip
	
	(3) If $q=0$, then the claim reduces to Theorem \ref{thm:one variable case} (3). Suppose $-1\leq q<0,~1\leq p\leq 2,~(p,q)\ne(1,-1)$ and $s\geq\frac{1}{p+q}$, then we apply \eqref{equ:variational method max-convex} to $(r_0,r_1,r_2)=(2t,2s,\frac{2}{-q})$ with $\frac{1}{t}=\frac{1}{s}-q$ and $(X,Y)=(A^{\frac{p}{2}}K,B^{\frac{q}{2}})$:
	\begin{equation}\label{equ:proof of thm psi 3}
	\Psi_{p,q,s}(A,B)=\max_{Z\in\bh^{\times}}\left\{\frac{s}{t}\Tr|A^{\frac{p}{2}}KZ|^{2t}+sq\Tr|B^{-\frac{q}{2}}Z|^{\frac{2}{-q}}\right\}.
	\end{equation}
	Since $sq<0$, $0<-q\leq 1$ and $t=\frac{1}{s^{-1}-q}\geq\frac{1}{p}$, we have by Theorem \ref{thm:one variable case} (1) and (3) that the maps 
	\[
	A\mapsto\frac{s}{t}\Tr|A^{\frac{p}{2}}KZ|^{2t}=\frac{s}{t}\Tr(Z^*K^*A^{p}KZ)^{t}
	\]
	and 
	\[
	B\mapsto sq\Tr|B^{-\frac{q}{2}}Z|^{\frac{2}{-q}}=sq\Tr(Z^*B^{-q}Z)^{\frac{1}{-q}}
	\]
	are both convex. Hence they are both jointly convex as functions in $(A,B)$ and so is $\Psi_{p,q,s}$ by Lemma \ref{lem:convex-concave} (1) and \eqref{equ:proof of thm psi 3}.
\end{proof}
\begin{rem}
	One can understand this step of reduction in the following heuristic way. In Figure \ref{figure:convex/concave}, the green region $[0,1]\times [0,1]$ is generated by two intervals of the $p$-axis and the $q$-axis: $[0,1]\times \{0\}$ and $\{0\}\times [0,1]$. That is how we deduce the joint concavity of $\Psi_{p,q,s}$ (Theorem \ref{thm:A} (1)) from the concavity of $\Upsilon_{p,s}$ (Theorem \ref{thm:one variable case} (1)) in the above proof. The proof of the yellow region of the Figure \ref{figure:convex/concave} can be understood in a similar way. 
\end{rem}

\textbf{Step 2:} In our second step we reduce Theorem \ref{thm:one variable case} to three particular cases.
\begin{thm}\cite{Epstein73,Hiai13concavity-I,CL08minkowski-II}\label{thm:reduction of convexity/concavity}
	Fix $K\in\bh^{\times}$, then 
	\begin{enumerate}
		\item $\Upsilon_{p,1/p}$ is concave when $0<p\le 1$ (Epstein);
		\item $\Upsilon_{p,s}$ is convex when $-1\le p<0$ and $0<s\le 1$ (Hiai);
		\item $\Upsilon_{p,1/p}$ is convex when $1\le p\le 2$ (Carlen-Lieb).
	\end{enumerate}
\end{thm}

\begin{proof}[Proof of Theorem \ref{thm:one variable case} given \ref{thm:reduction of convexity/concavity}]
    Indeed, when $0<p\le 1$, $0<s<\frac{1}{p}$ and $\frac{1}{s}=p+\frac{1}{t}$, by applying \eqref{equ:variational method min-concave} to $(r_0,r_1,r_2)=(2s,\frac{2}{p},2t)$ and $(X,Y)=(A^{\frac{p}{2}},K)$ we obtain that
	\[
	\Tr(K^{*}A^{p}K)^{s}=\min_{Z\in\bh^{\times}}\left\{sp\Tr(Z^{*}A^{p}Z)^{\frac{1}{p}}+\frac{s}{t}\Tr(K^{*}(Z^{-1})^*Z^{-1}K)^{t}\right\}.
	\]
	Then by Lemma \ref{lem:convex-concave} (1), the concavity of $\Upsilon_{p,1/p}$ implies the concavity of $\Upsilon_{p,s}$.
	
    \smallskip
    When $-1\le p<0$ and $s>1$, by applying \eqref{equ:variational method max-convex} to $(r_0,r_1,r_2)=(2,2s,\frac{2s}{s-1})$ and $(X,Y)=(A^{\frac{p}{2}},K)$ we obtain that 
	\begin{equation*}
	\Tr(K^{*}A^{p}K)^{s}
	=\max_{Z\in\bh^{\times}}\left\{s\Tr Z^{*}A^{p}Z-(s-1)\Tr(Z^{*}(K^{-1})^*K^{-1}Z)^{\frac{s}{s-1}}\right\}.
	\end{equation*}
	Then by Lemma \ref{lem:convex-concave} (1), the convexity of $\Upsilon_{p,1}$ implies the convexity of $\Upsilon_{p,s}$.
    
    \smallskip
	When $1\le p\le 2$, $s>\frac{1}{p}$ and $p=\frac{1}{s}+\frac{1}{t}$, by applying \eqref{equ:variational method max-convex} to $(r_0,r_1,r_2)=(\frac{2}{p},2s,2t)$ and $(X,Y)=(A^{\frac{p}{2}},K)$ we obtain that 
	\[
	\Tr(K^{*}A^{p}K)^{s}=\max_{Z\in\bh^{\times}}\left\{sp\Tr(Z^{*}A^{p}Z)^{\frac{1}{p}}-\frac{s}{t}\Tr(Z^{*}(K^{-1})^*K^{-1}Z)^{t}\right\}.
	\]
	Then by Lemma \ref{lem:convex-concave} (1), the convexity of $\Upsilon_{p,1/p}$ implies the convexity of $\Upsilon_{p,s}$.
\end{proof}	

\textbf{Step 3:} In the last step we reduce Theorem \ref{thm:reduction of convexity/concavity} to Lemma \ref{lem:Lieb-Ando}.

\begin{proof}[Proof of Theorem \ref{thm:reduction of convexity/concavity} given Lemma \ref{lem:Lieb-Ando}]
    The proof is inspired by the proof of (2) in \cite{CFL18conjecture}. Let us recall it first. If $s=1$, the convexity of $\Upsilon_{p,1}$ follows from the operator convexity of $A\mapsto A^p$ for $-1\le p<0$. If $0<s<1$, by applying \eqref{equ:variational method min-concave} to $(r_0,r_1,r_2)=(2s,2,\frac{2s}{1-s})$ and $(X,Y)=(A^{\frac{p}{2}}K,I)$, we have 
	\begin{equation*}
	\begin{split}
	\Tr(K^{*}A^{p}K)^{s}
	&=\min_{Z\in\bh^{\times}}\left\{s\Tr |A^{\frac{p}{2}}KZ|^2+(1-s)\Tr |Z^{-1}|^{\frac{2s}{1-s}}\right\}\\
	&=\min_{Z\in\ph^{\times}}\left\{s\Tr K^{*}A^{p}KZ+(1-s)\Tr Z^{\frac{s}{s-1}}\right\}\\
	&=\min_{Z\in\ph^{\times}}\left\{s\Tr K^{*}A^{p}KZ^{1-p}+(1-s)\Tr Z^{\frac{s(1-p)}{s-1}}\right\}.
	\end{split}
	\end{equation*}
	Since $\frac{s(1-p)}{s-1}<0$, the function $t\mapsto t^{\frac{s(1-p)}{s-1}}$ is convex. Thus $Z\mapsto \Tr Z^{\frac{s(1-p)}{s-1}}$ is convex (see for example \cite[Theorem 2.10]{Carlen10course}). This, together with Ando's convexity result (Lemma \ref{lem:Lieb-Ando} (2)) and Lemma \ref{lem:convex-concave} (2), yields the convexity of $\Upsilon_{p,s}$.
	
    \smallskip	
	
	Now we prove (1). There is nothing to prove when $p=1$. For $0<p<1$, by applying \eqref{equ:variational method max-convex} to $(r_0,r_1,r_2)=(2,\frac{2}{p},\frac{2}{1-p})$ and $(X,Y)=(A^{\frac{p}{2}}K,I)$, we have
		\begin{equation*}
		\begin{split}
		\Tr(K^{*}A^{p}K)^{\frac{1}{p}}
		&=\max_{Z\in\bh^{\times}}\left\{\frac{1}{p}\Tr |A^{\frac{p}{2}}KZ|^2-\frac{1-p}{p}\Tr |Z|^{\frac{2}{1-p}}\right\}\\
		&=\max_{Z\in\ph^{\times}}\left\{\frac{1}{p}\Tr K^{*}A^{p}KZ-\frac{1-p}{p}\Tr Z^{\frac{1}{1-p}}\right\}\\
		&=\max_{Z\in\ph^{\times}}\left\{\frac{1}{p}\Tr K^{*}A^{p}KZ^{1-p}-\frac{1-p}{p}\Tr Z\right\}.
		\end{split}
		\end{equation*}
	Then by Lieb's concavity result (Lemma \ref{lem:Lieb-Ando} (1)) and Lemma \ref{lem:convex-concave} (2), $\Upsilon_{p,1/p}$ is concave. 
	
	\smallskip
	(3) can be shown similarly. Indeed, the case $p=1$ is trivial. For $1<p\le 2$, by applying \eqref{equ:variational method min-concave} to $(r_0,r_1,r_2)=(\frac{2}{p},2,\frac{2}{p-1})$ and $(X,Y)=(A^{\frac{p}{2}}K,I)$, we have
	\begin{equation*}
	\begin{split}
	\Tr(K^{*}A^{p}K)^{\frac{1}{p}}
	&=\min_{Z\in\bh^{\times}}\left\{\frac{1}{p}\Tr |A^{\frac{p}{2}}KZ|^2+\frac{p-1}{p}\Tr |Z^{-1}|^{\frac{2}{p-1}}\right\}\\
	&=\min_{Z\in\ph^{\times}}\left\{\frac{1}{p}\Tr K^{*}A^{p}KZ+\frac{p-1}{p}\Tr Z^{\frac{1}{1-p}}\right\}\\
	&=\min_{Z\in\ph^{\times}}\left\{\frac{1}{p}\Tr K^{*}A^{p}KZ^{1-p}+\frac{p-1}{p}\Tr Z\right\}.
	\end{split}
	\end{equation*}
	Then by Ando's convexity result (Lemma \ref{lem:Lieb-Ando} (2)) and Lemma \ref{lem:convex-concave} (2), $\Upsilon_{p,1/p}$ is convex. 
\end{proof}

\begin{comment}
\begin{rem}
One can apply Lemma \ref{lem:convex-concave} directly to reduce the concavity of $\Upsilon_{p,s}, 0<p\le 1, 0<s\le \frac{1}{p}$ (Theorem \ref{thm:one variable case} (1)) to \emph{Lieb's Concavity Theorem} \cite{Lieb73WYD}: for $0\le p,q\le 1$ with $p+q\le 1$, and for any $K$, $\Psi_{p,q,1}(A,B)=\Tr K^*A^p KB^q$ is jointly concave. Analogously, Theorem \ref{thm:one variable case} (2) and (3) can be reduced directly to \emph{Ando's Convexity Theorem} \cite{Ando79}: for $-1\le q\le 0$ and $1\le p\le 2$ with $p+q\ge 1$, and for any $K$, $\Psi_{p,q,1}(A,B)=\Tr K^*A^p KB^q$ is jointly convex.
\end{rem}
\end{comment}

\begin{rem}
	Although the variational methods \eqref{equ:variational method min-concave} and \eqref{equ:variational method max-convex} admit analogues \eqref{equ:n-variables-min} and \eqref{equ:n-variable-max} of $n(\geq 3)$ variables, the joint convexity/concavity of 
	\[
	\ph^{\times}\times\cdots\times\ph^{\times}\ni (A_1,\dots,A_n)\mapsto\Tr(A_n^{\frac{p_n}{2}}K^{*}_{n-1}\cdots K^{*}_{1}A_1^{p_1}K_1\cdots K_{n-1}A_n^{\frac{p_n}{2}})^s
	\]
	can not be derived directly from Theorem \ref{thm:one variable case} because of the appearance of the term $\Tr|Z^{-1}_{j-1}X_jZ_j|^{r_j}$. For example, we have 
	\begin{equation}\label{equ:variation-three variables}
	\begin{split}
	&\Tr|X_{1}X_{2}X_{3}|^{r_0}\\
	=&\min_{Z_1,Z_2\in\bhx}\left\{\frac{r_0}{r_1}\Tr|X_{1}Z_{1}|^{r_1}+\frac{r_0}{r_2}\Tr|Z_{1}^{-1}X_{2}Z_{2}|^{r_2}+\frac{r_0}{r_3}\Tr|Z^{-1}_{2}X_{3}|^{r_3}\right\}.
	\end{split}
	\end{equation}
	To obtain the joint concavity of
	\[
	\ph^{\times}\times\ph\times\ph^{\times}\ni (A_1,A_2,A_3)\mapsto\Tr(A_3^{\frac{p_3}{2}}K^{*}_{2}A_2^{\frac{p_2}{2}}K^{*}_{1}A_1^{p_1}K_1A_2^{\frac{p_2}{2}}K_{2}A_3^{\frac{p_3}{2}})^s,
	\]
	via the variational method \eqref{equ:variation-three variables}, the concavity of the function of the form
	\[
	\ph^{\times}\ni A_2\mapsto\Tr|Y_{1}A_{2}^{\frac{p_2}{2}}Y_{2}|^{r_2}=\Tr(Y_2^*A_{2}^{\frac{p_2}{2}}Y_1^*Y_1A_{2}^{\frac{p_2}{2}}Y_2)^{\frac{r_2}{2}}
	\]
	is required. Unfortunately, little is known for general $Y_1^*Y_1\neq I$. Indeed, Carlen, Frank and Lieb proved that \cite[Corollary 3.3]{CFL16some} for $p,q,r\in\mathbb{R}\setminus\{0\}$, the function 
	\[
	(A,B,C)\mapsto\Tr C^{\frac{r}{2}}B^{\frac{q}{2}}A^{p}B^{\frac{q}{2}}C^{\frac{r}{2}}
	\]
	is never concave, and it is convex if and only if $q=2,~p,r<0$ and $-1\leq p+r<0$.
\end{rem}

\subsection*{Acknowledgement}
The author would like to thank Quanhua Xu, Adam Skalski, Ke Li and Zhi Yin for their valuable comments. He also would like to thank the anonymous referees for pointing out some errors in an earlier version of this paper and for helpful comments and suggestions that make this paper better. The research was partially supported by the NCN (National Centre of Science) grant 2014/14/E/ST1/00525, the French project ISITE-BFC (contract ANR-15-IDEX-03), NSFC No. 11826012, and the European Union's Horizon 2020 research and innovation programme under the Marie Sk\l odowska-Curie grant agreement No. 754411.
%and NSFC No. 11431011.

\bibliographystyle{alpha}
\bibliography{bibfile}

\end{document}